\documentclass[11pt]{article}

\usepackage{mathrsfs}
\usepackage{amssymb}
\usepackage{amsthm}
\usepackage{indentfirst}
\usepackage{amsmath}
\usepackage{hyperref}

\newtheorem{theoremalph}{Theorem}

\newtheorem{Proposition}[theoremalph]{Proposition}
\newtheorem*{Theorem A}{Theorem A}

\newtheorem{Lemma}{Lemma}[section]

\newtheorem{Remark}[Lemma]{Remark}

\newtheorem*{Acknowledgments}{Acknowledgments}

\begin{document}

\title{Expansive homoclinic classes}
\author{Dawei Yang and Shaobo Gan\\[1mm]
LMAM, School of Mathematical Sciences, \\
Peking University, Beijing 100871, P. R. China\\
{\tt yangdw@math.pku.edu.cn, gansb@math.pku.edu.cn}}

\maketitle

\begin{abstract}
We prove that for $C^1$ generic diffeomorphisms, every expansive
homoclinic class is hyperbolic.


\end{abstract}

\section{Introduction}

Let $M$ be a $d$-dimensional boundaryless Riemannian manifold.
Denote by ${\rm Diff}^1(M)$ the space of $C^1$ diffeomorphisms on
$M$, endowed with the usual $C^1$ topology. For $f\in{\rm
Diff}^1(M)$, a compact invariant set $\Lambda$ of $f$ is called
\emph{hyperbolic}, if there is a continuous invariant splitting
$T_\Lambda M=E^s\oplus E^u$ on $\Lambda$, and two constants $C\ge
1$, $\lambda\in(0,1)$, such that for any $x\in\Lambda$ and any
$n\in\mathbb{N}$, we have
$$\|Df^n|_{E^s(x)}\|\le C\lambda^n,\qquad
\|Df^{-n}|_{E^u(x)}\|\le C\lambda^n.$$ A periodic point $p$ of $f$
is {\emph hyperbolic} if the orbit ${\rm Orb}(p)={\rm Orb}_f(p)$ of
$p$ is a hyperbolic set. Denote by $\pi(p)$ the period of $p$.

A compact invariant set $\Lambda$ of $f\in{\rm Diff}^1(M)$ is called
\emph{expansive}, if there is $\alpha>0$, such that for any
$x,y\in\Lambda$, if $d(f^n(x),f^n(y))<\alpha$ for any
$n\in\mathbb{Z}$, then $x=y$. It is well-known that every hyperbolic
set is expansive.

A subset ${\cal R}\subset\rm{Diff}^1(M)$ is called \emph{residual},
if it contains a countable intersection of open and dense subsets of
$\rm{Diff}^1(M)$. A property is called ($C^1$) \emph{generic}, if it
holds in a residual subset of $\rm{Diff}^1(M)$. We use the
terminology ``for $C^1$ generic $f$'' to express ``there is a
residual subset ${\cal R}\subset{\rm Diff}^1(M)$ and $f\in{\cal
R}$''.

Given $f\in{\rm Diff}^1(M)$, two hyperbolic periodic points $p$ and
$q$ of $f$ are called \emph{homoclinically related} if $W^s({\rm
Orb}(p))\pitchfork W^u({\rm Orb}(q))\neq\emptyset$ and $W^u({\rm
Orb}(p))\pitchfork W^s({\rm Orb}(q))\neq\emptyset$, denoted as
$p\sim_f q$, or simply $p\sim q$. If $p$ and $q$ are homoclinically
related, their stable manifolds must have the same dimension. The
homoclinic class of $p$ is defined as $H(p)=\overline{\{q:~q\sim
p\}}$, which is a transitive compact invariant set of $f$. The study
of homoclinic classes is an important topic in smooth dynamical
systems, for instance,
\begin{itemize}
\item If $f$ is Axiom A, Smale's spectral
decomposition theorem says that the non-wandering set can be
decomposed into finitely many basic sets, and each basic set is a
homoclinic class.
\item The chain recurrent set can be divided into (maybe infinitely
many) chain recurrent classes. It was proved (\cite{BoC04}) that for
$C^1$ generic $f$, a chain recurrent class containing a periodic
point $p$ is the homoclinic class $H(p)$.
\end{itemize}

The following proposition is our main technical result.
\begin{Proposition}\label{Prop:noweak}
For $C^1$ generic $f$, if a homoclinic class $H(p)$ is expansive,
then there are three constants $\iota\in\mathbb{N}$, $K\ge 1$, and
$\lambda\in(0,1)$ such that for any periodic point $q\sim_f p$ with
period $\pi(q)>\iota$, and for any $x\in {\rm Orb}_f(q)$, we have
\begin{eqnarray*}
\prod_{j=0}^{[\pi(q)/\iota]-1}\|Df^{\iota}|_{E^s(f^{j\iota}(x))}\|
\le K\lambda^{[\pi(q)/\iota]}, \\
\prod_{j=0}^{[\pi(q)/\iota]-1}\|Df^{-\iota}|_{E^u(f^{j\iota}(x))}\|
\le K\lambda^{[\pi(q)/\iota]},\\[.1cm]
\|Df^\iota|_{E^s(x)}\|\cdot\|Df^{-\iota}|_{E^u(f^{\iota}(x))}\|\le\lambda.
\end{eqnarray*}
\end{Proposition}

The relation of expansiveness and hyperbolicity of homoclinic
classes has been discussed in \cite{PPSV08, PPV05, SaV06, SaV08}.
And it is essentially proved in \cite{SaV08} that under the
assumptions and conclusions of Proposition \ref{Prop:noweak}, the
homoclinic class $H(p)$ is hyperbolic. So, we get

\begin{theoremalph}
For $C^1$ generic $f$, every expansive homoclinic class of $f$ is
hyperbolic.
\end{theoremalph}

\section{Proof of Proposition \ref{Prop:noweak}}
We will first prove some new generic properties and then use these
generic properties and some known results to prove Proposition
\ref{Prop:noweak}. Let us introduce some terminologies first.

For $\eta>0$ and $f\in{\rm Diff}^1(M)$, a $C^1$ curve $\gamma$ is
called \emph{$\eta$-simply periodic curve of $f$} if
\begin{itemize}
\item $\gamma$ is diffeomorphic to $[0,1]$,
and its two endpoints are hyperbolic periodic points of $f$;
\item $\gamma$ is periodic with period $\pi(\gamma)$, i.e.,
$f^{\pi(\gamma)}(\gamma)=\gamma$, and ${\rm l}(f^i(\gamma))<\eta$
for any $0\le i\le \pi(\gamma)-1$, where ${\rm l}(\gamma)$ denotes
the length of $\gamma$;

\item $\gamma$ is normally hyperbolic.
(See \cite{HPS77} for the definition of normal hyperbolicity.)
\end{itemize}

Let $p$ be a periodic point of $f$. For $\delta\in(0,1)$, we say $p$
has a \emph{$\delta$-weak eigenvalue}, if $Df^{\pi(p)}(p)$ has an
eigenvalue $\sigma$ such that $(1-\delta)^{\pi(p)}<
|\sigma|<(1+\delta)^{\pi(p)}$.

The following lemma gives three generic properties, which says
(roughly) that if an arbitrary small perturbation of $f$ has some
``stable'' property, then ($C^1$ generic) $f$ itself has this
property.
\begin{Lemma}\label{Lem:newgeneric}
For $C^1$ generic $f$ and any hyperbolic periodic point $p$ of $f$,
\begin{enumerate}
\item for any $\eta>0$,
if for any $C^1$ neighborhood ${\cal U}$ of $f$, some $g\in{\cal U}$
has an $\eta$-simply periodic curve $\gamma$, such that the two
endpoints of $\gamma$ are homoclinically related with $p_g$, then
$f$ has an $2\eta$-simply periodic curve $\alpha$ such that the two
endpoints of $\alpha$ are homoclinically related with $p$.

\item for any $\delta>0$,
if for any $C^1$ neighborhood ${\cal U}$ of $f$, some $g\in{\cal U}$
has a periodic point $q\sim_g p_g$ with $\delta$-weak eigenvalue,
then $f$ has a periodic point $q'\sim_f p$ with $2\delta$-weak
eigenvalue.

\item for any $\delta>0$,
if for any $C^1$ neighborhood ${\cal U}$ of $f$, some $g\in{\cal U}$
has a periodic point $q\sim_g p_g$ with $\delta$-weak eigenvalue and
every eigenvalue of $q$ is real, then $f$ has a periodic point
$q'\sim_f p$ with $2\delta$-weak eigenvalue and every eigenvalue of
$q'$ is real.

\end{enumerate}
\end{Lemma}
\begin{proof}
Let ${\cal C}$ be the space of all compact subsets of $M$, endowed
with the Hausdorff distance. Then ${\cal C}$ is a compact separable
metric space. Let
$\mathscr{V}_1,\mathscr{V}_2,\cdots,\mathscr{V}_n,\cdots,$ be a
countable base of ${\cal C}$.

We prove item 1 first. We would like to introduce the terminologies
of two types of compact invariant sets so that item 2 and 3 can be
proved similarly. For any $\eta>0$, we call hyperbolic periodic
orbits \emph{type (I)}, and orbits of $\eta$-simply periodic curves
\emph{type (II$_\eta$)}. Both these two types of compact invariant
sets have continuations (\cite{HPS77}): if $\Lambda$ is a compact
invariant set of $f$ of type (I) or (II$_\eta$), then there is a
neighborhood ${\cal U}$ of $f$, a neighborhood $\mathscr{V}$ of
$\Lambda$ in ${\cal C}$, such that for any $g\in{\cal U}$, $\Lambda$
has a unique continuation $\Lambda_g\in\mathscr{V}$ of type (I) or
(II$_\eta$), and moreover, $\Lambda_g\to\Lambda$ as $g\to f$. We say
that a type (I) set $\Lambda^1$ and a type (II$_\eta$) set
$\Lambda^2$ have a relation $\leftrightsquigarrow_f$, if $\Lambda^1$
is homoclinically related with the two endpoints of $\Lambda^2$,
denoted by $\Lambda^1\leftrightsquigarrow_f\Lambda^2$. This relation
is stable: If $\Lambda^1\leftrightsquigarrow_f\Lambda^2$, then there
is a neighborhood ${\cal U}$ of $f$, such that for any $g\in{\cal
U}$, we have $\Lambda^1_g\leftrightsquigarrow_g\Lambda^2_g$.

Let ${\cal H}_{n}(\eta)$ be the set of $C^1$ diffeomorphisms $f$,
such that $f$ has a type (I) set $\Lambda^1\in\mathscr{V}_n$ and a
type (II$_\eta$) set $\Lambda^2$ verifying that
$\Lambda^1\leftrightsquigarrow_f\Lambda^2$. From the stability of
the relation $\leftrightsquigarrow$, ${\cal H}_{n}(\eta)$ is open.
Let ${\cal N}_{n}(\eta)={\rm Diff}^1(M)-\overline{{\cal
H}_{n}(\eta)}$.

Since ${\cal H}_{n}(\eta)\cup{\cal N}_{n}(\eta)$ is open and dense
in ${\rm Diff}^1(M)$ from their definitions, ${\cal
R(\eta)}=\cap_{n\in\mathbb{N}}({\cal H}_{n}(\eta)\cup{\cal
N}_{n}(\eta))$ is residual. And let ${\cal R}=\cap_{r\in\mathbb{Q},
r>0}{\cal R}(r)$, which is also a residual subset. We will prove
that if $f\in{\cal R}$, $f$ has the properties in item 1.

For any $\eta>0$, take $r\in\mathbb{Q}$ such that $\eta<r<2\eta$.
For any hyperbolic periodic point $p$ of $f$, take a neighborhood
$\mathscr{V}_n$ of ${\rm Orb}_f(p)$ in ${\cal C}$, such that ${\rm
Orb}_f(p)$ is the unique compact invariant set belonging to
$\mathscr{V}_n$. Since $\eta<r$, according to the assumption of item
1, for any $C^1$ neighborhood ${\cal U}$ of $f$, some $g\in{\cal U}$
has a $r$-simply periodic curve $\gamma$, such that the two
endpoints of $\gamma$ are homoclinically related with $p_g$. If the
neighborhood ${\cal U}$ of $f$ is small enough, the orbit of the
continuation $p_g$ is contained in $\mathscr{V}_n$ for any
$g\in{\cal U}$. This implies that $f\not\in{\cal N}_n(r)$ and hence
$f\in{\cal H}_n(r)\subset{\cal H}_n(2\eta)$, which means that $f$
has a periodic orbit $\mathscr{O}\in\mathscr{V}_n$ and a
$2\eta$-simply periodic curve, such that its endpoints are
homoclinically related with $\mathscr{O}$. By the choice of
$\mathscr{V}_n$, this periodic orbit $\mathscr{O}$ is ${\rm
Orb}_f(p)$. This finishes the proof of item 1.

We can prove item 2 (and item 3) similarly by defining the type
(II$_\delta$) sets to be periodic orbits homoclinically related with
$p$ which have $\delta$-weak eigenvalue (and every eigenvalue is
real).
\end{proof}

\begin{Remark}\label{Rem: noweak}
According to item 2 of Lemma \ref{Lem:newgeneric}, for $C^1$ generic
$f$ and any $\delta>0$, for any hyperbolic periodic point $p$ of
$f$, if every periodic point $q\sim_f p$ has no $2\delta$-weak
eigenvalue, then there is a neighborhood $\cal{U}$, such that for
any $g\in\cal{U}$, any periodic point $q\sim_g p_g$ has no
$\delta$-weak eigenvalue.
\end{Remark}
\begin{Lemma}\label{Lem:realweak}
For $C^1$ generic $f$, for any hyperbolic periodic point $p$ of $f$,
if for any $\delta>0$, $f$ has a periodic point $q\sim_f p$ with
$\delta$-weak eigenvalue, then $f$ has a periodic point $p_1\sim_f
p$ with $\delta$-weak eigenvalue, whose eigenvalues are all real.
\end{Lemma}
\begin{proof}
Assume that  for some periodic point $q\sim_f p$ with $\delta$-weak
eigenvalue, $Df^{\pi(q)}(q)$ has some complex eigenvalues. As in the
proof of \cite[Lemma 4.16]{BDP03}, an arbitrarily small perturbation
$g$ of $f$ has a periodic point $p_1\sim_g q_g$ with $\delta$-weak
eigenvalue and all eigenvalues of $p_1$ are real. Since $p_1\sim_g
q_g$ and $q_g\sim_g p_g$ imply that $p_1\sim_g p_g$, according to
item 3 of Lemma \ref{Lem:newgeneric}, $C^1$ generic $f$ has this
property itself.
\end{proof}

\begin{proof}[\bf Proof of Proposition \ref{Prop:noweak}]
For $C^1$ generic $f$, assume that $H(p)$ is an expansive homoclinic
class. We first claim that there is $\delta_0>0$, such that for any
periodic point $q\sim p$, $q$ has no $\delta_0$-weak eigenvalue.
Otherwise, for any $\delta>0$, $H(p)$ contains a periodic point
$q_\delta\sim_f p$ with $\delta$-weak eigenvalue. Then by Lemma
\ref{Lem:realweak}, $H(p)$ contains a periodic point
$q_\delta'\sim_f p$ with $\delta$-weak eigenvalue, and whose
eigenvalues are real. From \cite[Section 4]{SaV06}, for any
$\eta>0$, for any $C^1$ neighborhood ${\cal U}$ of $f$, some
$g\in{\cal U}$ has an $\eta$-simply periodic curve $\gamma$, whose
endpoints are homoclinically related with $p_g$. By item 1 of Lemma
\ref{Lem:newgeneric}, for any $\eta>0$, $f$ itself has a
$2\eta$-simply periodic curve $\alpha$, such that the endpoints of
$\alpha$ are homoclinically related with $p$. Then all iterates of
the two endpoints of $\alpha$ have distance $<2\eta$, which
contradicts the expansiveness of $H(p)$.

According to Remark \ref{Rem: noweak}, (for $C^1$ generic $f$),
there is a $C^1$ neighborhood $\cal U$ of $f$, such that for any
$g\in{\cal U}$, any periodic point $q\sim_g p_g$ has no
$\delta_0/2$-weak eigenvalue.

Gourmelon proved (\cite[Theorem 2.1]{Gou08}) an extension of Franks
Lemma (\cite{Fra71,Man82}), which preserves the (un)stable
manifolds: For any neighborhood ${\cal U}$ of $f$, there is
$\varepsilon>0$, such that for any hyperbolic periodic point $q$ of
$f$, for any sequence of linear isomorphisms $\{L_i: T_{f^iq}M\to
T_{f^{i+1}q}M\}_{i=0}^{\pi(q)-1}$ verifying
$\|Df(f^i(q))-L_i\|<\varepsilon$ for $0\le i\le \pi(q)-1$, there
exists $\delta_1>0$, for any $\delta_2\in (0, \delta_1]$, there
exists $g\in\cal U$, such that $f^i(q)=g^i(q)$, $Dg(f^i(q))=L_i$,
and $g=f$ outside of $B_{\delta_2}({\rm Orb}_f(q))$; moreover, if
$y\in W^s(f^iq)-B_{\delta_2}(f^iq)$ for some $0\le i\le \pi(q)-1$,
and $f^ky \in B_{\delta_2}(f^iq)$ for some $k>0$ implies $f^k y\in
W^s_{\delta_2}(f^iq)$, then $y\in W^s(g^iq, g)$, where
$$
W^s_{\delta_2}(x)=\{z\in M| \forall n\ge 0,\ d(f^{n}x,
f^{n}z)\le\delta_2\}
$$
is the local stable manifold of $x$ with size $\delta_2$ with
respect to $f$ and $W^s(x,g)$ is the stable manifold of $x$ with
respect to $g$. Similar conclusion also holds for the unstable
manifolds. So, we can give a perturbation of the derivatives along a
periodic orbit which preserves the homoclinic relation
simultaneously.

Since there is a $C^1$ neighborhood $\cal U$ of $f$, such that for
any $g\in{\cal U}$, any periodic point $q\sim_g p_g$ has no
$\delta_0/2$-weak eigenvalue, according to the extension of Franks'
lemma described above,
$$\{Df(q), Df(f(q)),\cdots,Df(f^{\pi(q)-1}(q)):q\sim_f p\}$$
is a uniformly hyperbolic family of periodic sequences of
isomorphisms of ${\mathbb R}^d$ (see \cite[Page524-525]{Man82} for
more details). By \cite[Lemma II.3]{Man82}, there are three
constants $\iota\in\mathbb N$, $K\ge 1$, and $\lambda\in(0,1)$, such
that for any periodic point $q\sim_f p$ with period $\pi(q)>\iota$,
for any $x\in{\rm Orb}_f(q)$, we have
\begin{eqnarray*}
\prod_{j=0}^{[\pi(q)/\iota]-1}\|Df^{\iota}|_{E^s(f^{j\iota}(x))}\|
\le K\lambda^{[\pi(q)/\iota]}, \\
\prod_{j=0}^{[\pi(q)/\iota]-1}\|Df^{-\iota}|_{E^u(f^{j\iota}(x))}\|
\le K\lambda^{[\pi(q)/\iota]},\\[.1cm]
\|Df^\iota|_{E^s(x)}\|\cdot\|Df^{-\iota}|_{E^u(f^{\iota}(x))}\|\le\lambda.
\end{eqnarray*}
\end{proof}

\begin{Acknowledgments}
We would like to thank Prof. L. Wen for useful discussions. Remark
\ref{Rem: noweak} was pointed out to us by Christian Bonatti, for
whom we should thank a lot. We also want to thank Martin Sambarino,
who kindly explained his work with J. Vieitez to us at Trieste. This
work is supported by NSFC (10531010) and MOST (2006CB805903).
\end{Acknowledgments}

\end{document}